\documentclass{amsart}
\title{Asymptotics for the number of simple $(4a+1)$-knots of genus 1}
\author{Alison Beth Miller}
\address{Mathematical Reviews\\ 416 Fourth Street, Ann Arbor, MI 48103}
\email{alimil@umich.edu}
\usepackage{amsmath, amsthm, amssymb, amsopn, amsfonts, amscd}
\usepackage{delarray}
\usepackage{enumerate}
\usepackage{fullpage}

\newtheorem{thm}{Theorem}[section]
\newtheorem{lemma}[thm]{Lemma}
\newtheorem{prop}[thm]{Proposition}
\newtheorem{cor}[thm]{Corollary}

\newtheorem{heuristic}{Heuristic}
\newtheorem*{heuristic*}{Heuristic}

\theoremstyle{definition}
\newtheorem*{defn}{Definition}

\theoremstyle{remark}
\newtheorem*{rem}{Remark}
\newtheorem*{question}{Question}

\DeclareMathOperator{\Aut}{Aut}

\DeclareMathOperator{\Frac}{Frac}

\DeclareMathOperator{\tr}{tr}

\DeclareMathOperator{\Cl}{Cl}

\DeclareMathOperator{\lk}{lk}

\DeclareMathOperator{\vol}{vol}

\DeclareMathOperator{\content}{content}
\DeclareMathOperator{\Disc}{Disc}

\newcommand{\isom}{\cong}

\newcommand{\Wedge}{\bigwedge}

\def\to{\rightarrow}

\def\p{\mathfrak{p}}

\def\R{\mathbb{R}}

\def\Q{\mathbb{Q}}

\def\Z{\mathbb{Z}}

\def\cO{\mathcal{O}}

\DeclareMathOperator{\SL}{SL}
\DeclareMathOperator{\Sp}{Sp}

\newcommand{\into}{\hookrightarrow}

\begin{document}
\maketitle
\begin{abstract}
We investigate the asymptotics of the total number of simple $(4a+1)$-knots with Alexander polynomial of the form $mt^2 +(1-2m) t + m$ for some nonzero $m \in [-X, X]$. Using Kearton and Levine's classification of simple knots, we give equivalent algebraic and arithmetic formulations of this counting question.  In particular, this count is the same as the total number of $\Z[1/m]$-equivalence classes of binary quadratic forms of discriminant $1-4m$, for $m$ running through the same range.   Our heuristics, based on the Cohen-Lenstra heuristics, suggest that this total is asymptotic to $X^{3/2}/\log X$, and the largest contribution comes from the values of $m$ that are positive primes.  Using sieve methods, we prove that the contribution to the total coming from $m$ positive prime is bounded above by $O(X^{3/2}/\log X)$, and that the total itself is $o(X^{3/2})$.

\end{abstract}

\section{Introduction}

In this paper we will count  simple $(4a+1)$-knots by way of invariants with arithmetic structure.   Informally, an $n$-knot is a knotted copy of $S^n$ in $S^{n+2}$: for a formal definition see Section~\ref{knot background}.  For $n \ge 5$, it is impossible to classify all $n$-knots, but there is a restricted family, the simple $n$-knots, which have been completely classified by classical algebraic invariants.  In this paper we will look at the case $n \equiv 1 \pmod 4$ and $n >1$.  The case where $n \equiv 3 \pmod 4$ has many similarities, but also a few differences  and would be an interesting subject for further research.

There is a natural definition of the genus for $(4a+1)$-knots.  For $a \ge 1$ this definition has the property that the degree of the Alexander polynomial of a genus $g$ simple $(4a+1)$-knot is precisely $2g$ (though this is is very far from being true in general for the familiar low-dimensional case $a = 0$).

\subsection{Questions and Heuristics}
We are interested in the general question: for a given positive integer $g$, what are the asymptotics of the number of distinct $(4a+1)$-knots with squarefree Alexander polynomial of degree $2g$ and coefficients of  ``size'' bounded by $X$?  It is known by results of Bayer and Michel~\cite{bayer-michel-finitude} that this number is always finite.  In this paper we address the case of $g = 1$.   In this case, the only possible Alexander polynomials are of the form $\Delta_m = m t^2 + (1- 2m)t + m$ for $m \in \Z$ and the question becomes:

\begin{question}[Counting question, knot version]
Fix $a \ge 1$.
Asymptotically, what is the total number of equivalence classes of simple $(4a+1)$-knots of genus $1$ with Alexander polynomial of the form $\Delta_m$ for $|m| \le X$?
\end{question}

Perhaps surprisingly, the answer to this question does not depend on the value of $a$, but is uniform for all $a \ge 1$.  This follows from the classification of simple knots in terms of their Alexander module and Blanchfield pairing \cite{kearton-simple}.  Using this classification, we can transform our counting question into an equivalent question which is entirely algebraic: 

\begin{question}[Counting question, equivalent Alexander/Blanchfield version]
Fix $a \ge 0$.
Asymptotically, what is the total number of isomorphism  classes are there of Alexander modules with Blanchfield pairing coming from simple $(4a+1)$-knots with Alexander polynomial of the form $\Delta_m$ for $|m| \le X$?
\end{question}
Note that this question has a uniform answer for all $a \ge 0$.

We will show that the question of counting knots is also equivalent to the following statement in number theory:
\begin{question}[Counting question, equivalent quadratic form version]
Asymptotically, what is the total, over all $m$ with $|m| \le X$, of the number of $\SL_2(\Z[\frac{1}{m}])$-equivalence classes of binary quadratic forms over $\Z[\frac{1}{m}]$ having discriminant $1-4m$?
\end{question}

The asymptotic behavior of the number of equivalence classes of binary quadratic forms with bounded discriminant has been studied by many mathematician, starting with Gauss.  The difficulty in Question 1 comes from inverting the varying integer $m$. 
We will see  that there is a clear split in quantitative behavior between the cases where the constant term of the Alexander polynomial is a positive prime and the cases where it is composite or negative.  Using the theory of binary quadratic forms, we propose the following heuristics.   In Section~\ref{heuristics} we justify these heuristics using the Cohen-Lenstra-Hooley heuristics for binary quadratic forms plus some additional reasonable assumptions.

\begin{heuristic}\label{prime heuristic}
The total number of distinct Alexander modules (with pairing) having Alexander polynomial equal to  $\Delta_p$ for some prime $p$ in the range $[1, X]$ is asymptotic to a constant times $X^{3/2}/\log X$.
\end{heuristic}

\begin{heuristic}\label{negative prime heuristic}
The total number of distinct Alexander modules (with pairing) having Alexander polynomial equal to  $\Delta_{-p}$ for some prime $p$ in the range $[1, X]$ is asymptotic to a constant times $X \log X$.
\end{heuristic}

\begin{heuristic}\label{not prime}
The total number of distinct Alexander modules (with pairing) having Alexander polynomial equal to  $m t^2 + (1-2m)t + m$ for $m$ running over all integers in the range $[-X, X]$ with $|m|$ not prime is asymptotic to a constant times $X \log X$.
\end{heuristic}

%Of these heuristics, Heuristic~\ref{not prime} is the most tenuous and warrants further investigation.

\subsection{Results}
The difficulty in proving these heuristics is that we are in general counting quadratic forms over rings with infinite unit group.  However, the total contribution from $m$ prime and positive can be bounded above using Rosser's sieve.

\begin{thm}\label{prime upper}

The total number of isotopy classes of simple $(4a+1)$-knots having Alexander polynomial equal to  $p t^2 + (1-2p)t + p$ for $p$ running over all primes in the range $[1, X]$ is (unconditionally) bounded above by $O(X^{3/2}/\log X)$ (that is, it is $\ll X^{3/2}/\log(X)$).
\end{thm}
\begin{rem}
In this paper, we use both Vinogradov's notation and big O notation as convenient: both $a \ll b$ and $a = O(b)$ mean that $a(X)$ is at most a constant times $b(X)$ for sufficiently large $X$.  We also use $a \asymp b$ to mean that $a$ has the same asymptotic order of growth as $b$: that is, $a \ll b$ and $b \ll a$.
\end{rem}

We expect that the asymptotic in Theorem~\ref{prime upper} is sharp.  However, our sieve methods are not powerful enough to give a corresponding lower bound. Instead we will be able to apply the Brauer-Siegel theorem pointwise to obtain

\begin{thm}\label{prime lower}
The total number of isotopy classes of simple $(4a+1)$-knots having Alexander polynomial equal to  $\Delta_p$ for $0 \le p \le X$ is $\gg (X^{3/2-\epsilon})$ for all $\epsilon > 0$.
\end{thm}
Like the Brauer-Siegel theorem, this result is ineffective.

\begin{rem}
Both of these results should generalize to $g > 1$.  To prove the upper bound, we will need to replace Gauss's asymptotics for binary quadratic forms with asymptotics for $\Sp_{2g}$-orbits on $2g$-ary quadratic forms.  In a forthcoming paper \cite{miller_sp2g}, the author obtains such bounds on $\Sp_{2g}$-orbits, and she hopes to apply these bounds to knot theory in future work.

For the lower bound, it should be possible to replace the Brauer-Siegel theorem with a relative Brauer-Siegel theorem for CM extensions.   However there may be technical issues regarding the cases of non-maximal orders.
\end{rem}

For the cases where $m$ is not a positive prime, it is much harder to obtain sharp results.  For instance, in the case of $m = -p$, Theorem~\ref{prime upper} remains true with exactly the same proof, but Heuristic~\ref{negative prime heuristic} suggests a strictly lower order of growth. 

However, we will show the following weak upper bound on the total:
\begin{thm}\label{aggregate}
The total number of distinct Alexander modules (with pairing) having Alexander polynomial equal to  $m t^2 + (1-2m)t + m$ for $m$ running over all integers in the range $[-X, X]$ is bounded above by $o(X^{3/2})$.
\end{thm}
We do not find an explicit rate of growth for our upper bound, but the proof indicates that the ratio implicit in the little $o$ notation should go to $0$ very slowly.  In particular, this is a much worse bound than the $O(X^{3/2}/\log X)$ bound in Theorem~\ref{prime upper}. 

\subsection{Acknowledgments}
I thank Manjul Bhargava for leading me to explore this problem, and for his guidance and support throughout the process.  Thanks also to Barry Mazur for helpful advice on the organization of this paper, and Arul Shankar, Alexander Smith, Lenny Ng for comments on various drafts.  I also thank the anonymous referee for detailed feedback that greatly improved the exposition of the paper.  Much of this paper was originally part of the author's Ph.D. thesis at Princeton University, where she was supported by an NSF Graduate Fellowship and an NDSEG Fellowship.

\section{The connection between knot theory and arithmetic}
\subsection{Preliminaries on knots}\label{knot background}
There are multiple different ways to formalize the notion of an $n$-knot.  The following two definitions are both commonly used:
\begin{defn}
\begin{enumerate}[(i)]
\item An {\em $n$-knot $K$} is a PL (piecewise linear) topologically embedded copy of $S^n$ in $S^{n+2}$ that is locally flat (locally homeomorphic to $\R^n \subset \R^{n+2}$).  Equivalence is given by ambient isotopy.

\item An {\em $n$-knot} is a smoothly embedded submanifold $K$ of $S^{n+2}$ that is homeomorphic to $S^n$ (but not necessarily diffeomorphic; $K$ might be an exotic sphere).  Equivalence is induced by orientation-preserving diffeomorphisms of the $S^{n+2}$.
\end{enumerate}

In both cases, we will consider both $S^n$ and $S^{n+2}$ as oriented manifolds, so that reversing the orientation of either or both may give an inequivalent knot.
\end{defn}

Although it is far from obvious, the classification results we use will give the same answer regardless of which of the two definitions above is used. I will talk about knots and equivalence with the understanding that all statements hold using either formulation.

\begin{defn}
The knot $K$ is called {\em simple} if $\pi_i(S^{n+2} - K) = \pi_i(S^{1})$ for all $i \le (n-1)/2$.
\end{defn}

For a simple $(4a+1)$-knot, we define the {\em Alexander module} of $K$ as $H^{2a+1}(Y; \Z)$ where $Y$ is the (unique) infinite cyclic cover of the knot complement $S^{n+2} \setminus K$ (one can show this is the only nontrivial homology group of $Y$).  The Alexander module is naturally a module over $\Z[t, t^{-1}]$, and also carries a natural duality pairing, known as the Blanchfield pairing.  In the case $a = 0$ this is the standard definition of the Alexander module of a (1-)knot. 

The Alexander module is annihilated by the Alexander polynomial $\Delta_K(t)$, another fundamental knot invariant, so is naturally a module over the quotient $\Z[t, t^{-1}] / \Delta_K(t)$: this gives a first indication that there is something interesting going on arithmetically.  

The Alexander module, and particularly the Blanchfield pairing, can be unwieldly to work with.  However, there is an easier route to the information carried in these invariants, via Seifert hypersurfaces and Seifert pairings.   

\begin{defn}
A {\em Seifert hypersurface} in $S^{n+2}$ is a compact oriented $(n+1)$-manifold $V$ with boundary such that $K = \partial V$ is homeomorphic to $S^n$.  We say that $V$ is a Seifert hypersurface for the knot $K$.
\end{defn}
(We again have the choice of working in either the PL category or the smooth category, and again all results we use hold uniformly in both cases.)

\begin{defn} 
A {\em simple Seifert hypersurface} in $S^{n+2}$ is said to be {\em simple} if $V$ is $\frac{n-1}{2}$-connected.
\end{defn}

It is known that the simple knots are exactly those with simple Seifert hypersurfaces.

\begin{thm}
If $V$ is a simple Seifert hypersurface in $S^{n+2}$, then $\partial V$ is a simple $n$-knot.  Conversely, any simple $n$-knot $K$ has a (non-unique) simple Seifert hypersurface.
\end{thm}
\begin{proof}
Farber states this as Theorem~0.5 in \cite{farber-simple}, where he deduces it from results of Levine~\cite{levine-unknotting,  levine-knot_modules} and Trotter~\cite{trotter-s-equivalence}.
\end{proof}

  If $V$ is a simple Seifert hypersurface in $S^{4a+3}$ it follows from the Hurewicz theorem and Poincare duality that $H_i(V, \Z)$ is trivial for all $i$ except $i = 2a+1$, and that $H_{2a+1}(V, \Z)$ is a free $\Z$-module of even rank.  We now recall the standard definition of the the Seifert pairing on the homology of $V$.
  
\begin{defn} 
 The {\em Seifert pairing} is the pairing  on $H_{2a+1}(V, \Z)$  defined by $\langle \alpha, \beta \rangle = \lk(\alpha^+, \beta)$ where $\lk$ denotes the linking (Alexander duality) pairing $H_{2a+1}(V; \Z) \times H_{2a+1}(S^{4a+3} \setminus V; \Z) \to \Z$, and $\alpha^+ \in { H_{2a+1}(S^{4a+3} \setminus V; \Z)}$ is obtained by pushing $\alpha$ off $V$ in the positive normal direction.
 
 The Seifert matrix $P$ of $V$ is the matrix of this a pairing with respect to a basis of $H_{2a+1}(V, \Z)$. 
 \end{defn}
   
  The Seifert pairing is neither symmetric nor skew-symmetric, so the same is true of its matrix $P$.  However, the skew-symmetric part of the Seifert pairing is equal to the intersection pairing, so we know that ${\det (P - P^t) = 1}$.  This turns out to be the only constraint on $P$.

\begin{defn}
If $V$ is a simple Seifert surface in $S^{4a+3}$, we define the genus of $V$ to be half the rank of $H_{2a+1}(V, \Z)$.

The genus of a $(4a+1)$-knot $K$ is the minimum genus of any Seifert surface for $K$.
\end{defn}

 In the classical low-dimensional case $a =0$, the genus is a subtle geometric invariant of $1$-knots.  However for  $a > 1$, the genus of a simple knot $K$ is easily computed from the Alexander polynomial of $K$, which in turn can be computed from any Seifert matrix.
 
 \begin{thm}\cite{levine-polynomial}\label{alexander poly formula}
 If $V$ is any simple Seifert surface for $K$, and $P$ is a Seifert matrix for $V$, then 
 \begin{equation} 
 \Delta_K = t^{-\dim \ker P} (\det t P - P^t)
 \end{equation}
 is the normalized Alexander polynomial of $K$, and in particular is independent of $V$.
 \end{thm}
 \begin{cor}
 If $K$ is any simple knot, then the genus of $K$ is equal to $\frac{1}{2} \deg \Delta_K$.
 \end{cor}
\begin{proof}
This follows from Theorem~\ref{alexander poly formula} and the result that any simple knot has a Seifert surface with nondegenerate Seifert pairing (itself a consequence of Theorem 3 and Proposition 1 in \cite{levine-classification}).
\end{proof}
\subsection{Relationship between knots, ideal classes, and quadratic forms}

We now use the classification of simple knots to show that the various forms of Question 1 stated in the introduction are equivalent. 
\begin{thm}[Classification of simple knots]\label{classification}\cite{ levine-classification, levine-knot_modules,trotter-knot_modules, kearton, farber-simple}
The following are in bijection with each other:
\begin{enumerate}[(i)] \item equivalence classes of simple $(4a+1)$-knots of genus $\le g$ 
\item $S$-equivalence classes of $2g\times 2g$ Seifert matrices $P$ \cite{levine-classification}
\item Alexander modules of genus $\le g$ equipped with Blanchfield pairing. \cite{levine-knot_modules, trotter-knot_modules}
\item $R$-equivalence classes of $\Z[z]$-modules with isometric structures \cite{farber-simple}
\end{enumerate}
\end{thm}

\begin{rem}
For each of these objects, there is a natural way of defining an Alexander polynomial. The easiest one to define is that for Seifert matrices, where the Alexander polynomial equals $t^{-\dim \ker P} (\det t P - P^t)$, as in Theorem~\ref{alexander poly formula}.  For an Alexander module (with pairing), the Alexander polynomial is defined as a normalized generator of the zeroth Fitting ideal.

These bijections preserve the Alexander polynomial.
\end{rem}

For $g = 1$, we can add two more items to the list, which we now introduce.  
First define \[R_m = \Z  [t, t^{-1}]/\Delta(t) \isom \Z \left[\frac{1}{m}, \frac{1 + \sqrt{1-4m}}{2} \right].\]  (Note that this ring is called $\cO_m$ in the author's Ph.D thesis \cite{miller_thesis}.)

Inspired by Bhargava's defnition in \cite{bhargava-higher_composition_i} we define
\begin{defn}
An {\em oriented ideal class} of $R_m$ is a homothety class of fractional ideals $I$ of $R_m$ equipped with an isomorphism $\phi: \Wedge^2 I \isom \Z[\frac{1}{m}]$ of $\Z[\frac{1}{m}]$-modules.
\end{defn}

Any such $\phi$ can be written in the form \[\phi(\alpha, \beta) = \tr\left(\frac{\alpha \beta}{\kappa \sqrt{1-4m}}\right)\] for a unique $\kappa \in \Z[\frac{1}{m}]$.  Such a $\phi$ maps $\Wedge^2 I$ isomorphically to $\Z[\frac{1}{m}]$ if and only if $\kappa$ is a generator of the fractional ideal $N I$ of $\Z[\frac{1}{m}]$.

Hence we can also describe oriented ideal classes of $R_m$ as equivalence classes $[I, \kappa]$ of pairs where $I$ is a fractional ideal of $R_m$ and $\kappa \in \Z[\frac{1}{m}]$ is a generator of $N I$, modulo the equivalence relation $(I, \kappa) \sim (\alpha I, N \alpha \kappa)$ for every $\alpha$ in $\Frac(R_m) = \Q(\sqrt{1-4m})$.  

In this paper we will write our oriented ideal classes as $[I, \kappa]$.  We will use the same notation for imaginary quadratic rings.

Not all oriented ideal classes of $R_m$ come from invertible ideals, but the ones that do form a group, which we denote $\Cl^+(R_m)$, and call the {\em oriented class group} of $\Cl^+(R_m)$.  The {\em oriented class group} is also sometimes called the ``narrow class group'', as in Bhargava \cite{bhargava-higher_composition_i}.  However, it is more common to use ``narrow class group'' to refer to the quotient of the group of fractional ideals by the totally positive principal fractional ideals.

These two notions of oriented and narrow class group are closely related.  For instance, if $\cO$ is a quadratic order, we can define the oriented class group $\Cl^+(\cO)$ analogously (replacing $\Z[\frac{1}{m}]$ with $\Z$).  When $\cO$ is a real quadratic order then $\Cl^+(\cO)$ is then precisely the narrow class group of $\cO$, and when $\cO$ is imaginary quadratic, the natural map $\Cl^+(\cO) \to \Cl(\cO)$ is a surjective 2-1 map.  However, for our quadratic rings $R_m$, where the base ring $\Z[\frac{1}{m}]$ has a potentially large unit group, in general the oriented class group $\Cl^+(R_m)$ will be much larger than the narrow class group of $R_m$, as one ideal class may have many inequivalent orientations.

\begin{rem}
In her Ph.D. thesis \cite{miller_thesis} the author defined conjugate self-balanced modules and ideal classes, which generalize the definition of oriented ideal classes above to the case of $g > 1$.  They can be thought of as relative ideal classes for a quadratic extension of rings.
\end{rem}

\begin{thm}
When $g =1$ we can add the following two items to the list in Theorem~\ref{classification}:

\begin{enumerate}[(v)]

\item[(v)] pairs $(m, [I, \kappa])$ where $m$ is an integer and $[I, \kappa]$ is an {\em oriented ideal class} of the ring $R_m = \Z[\frac{1}{m}, \frac{1 + \sqrt{1-4m}}{2}]$.

\item[(vi)] pairs $(m, [Q])$ where $m$ is an integer and $[Q]$ is an $\SL_2[\Z[\frac{1}{m}]]$-equivalence class of binary quadratic forms over $\Z[\frac{1}{m}]$ with $\Disc Q = 1-4m$.

In both cases, the corresponding Alexander polynomial is \[\Delta_m = m t^2 + (1-2m) t + m.\]
\end{enumerate}

\end{thm}
\begin{proof}

First of all, the bijection between (v) and (vi) is a generalization of the standard bijection between binary quadratic forms and ideal classes in quadratic rings.  Given an oriented ideal class $(I, \phi)$ of $R_m$, choose any $\Z[\frac{1}{m}]$ basis $u_1, u_2$ of $I$ with $\phi(u_1 \wedge u_2) = 1$.  The function $\phi((\sqrt{1-4m}) a \wedge b)$ is a symmetric bilinear form on the rank $2$ $\Z[\frac{1}{m}]$-module $I$.  If we write $\phi$ out in the basis $u_1, u_2$ we obtain a binary quadratic form $Q$ of discriminant $\sqrt{1-4m}$.

To finish, it is easiest to either biject Alexander modules with oriented ideal classes, or Seifert matrices with binary quadratic forms.  The former bijection is easier to prove, the latter easier to describe.  We prove the former:

If $M$ is an Alexander module with Alexander polynomial $\Delta_m$, we can view $M$ as a module over the quotient ring $R_m = \Z[t, t^{-1}]/\Delta_m$.  Because $\Delta_m$ is squarefree, we have $M \otimes_\Z \Q \isom R_m \otimes_\Z \Q$ as $ R_m \otimes_\Z \Q = \Q[t, t^{-1}]/\Delta(t)$-modules.  Hence $M$ is isomorphic as $R_m$-module to some fractional ideal $I$ of $R_m$.  Choose such an $I$ and an isomorphism $\phi: I \to M$.

To put an orientation on $I$, we use the Blanchfield pairing.  The isomorphism $\phi$ lets us transfer the Blanchfield pairing on $M$ to a $R_m$-hermitian perfect pairing 
\[
\langle, \rangle : I \times I \to \frac{1}{\Delta(t)}\Z[t, t^{-1}] / \Z[t, t^{-1}].
\]
To obtain an orientation, we compose with the map $T: \frac{1}{\Delta(t)}\Z[t, t^{-1}]/ \Z[t, t^{-1}] \to \Q$ sending $[f] \mapsto f'(0)$.  The map $T$ is a special case of ``Trotter's trace function'' in knot theory.  It's determined $\Z[\frac{1}{m}]$-linearly from the values  $T(\frac{1}{\Delta(t)}) = 0$, $T(\frac{t}{\Delta(t)}) = \frac{1}{m}$.

By standard arguments, the pairing $T(\langle a, b \rangle): I \times I \to \Z[\frac{1}{m}]$ is skew-symmetric with determinant a unit in $\Z[\frac{1}{m}]$.   Moreover we can recover the pairing $\langle, \rangle$ from the pairing $T(\langle a, b \rangle)$.
\end{proof}

We can now apply the formulas for the Alexander module and Blanchfield pairing in terms of the Seifert matrix  original given by Levine in \cite{levine-knot_modules} and reproved by Friedl and Powell in \cite{friedl-powell-blanchfield}.  When one works out the details, it turns out that the composite map from Seifert matrices to binary quadratic forms sends a Seifert matrix $P$ to the quadratic form with matrix $\frac{P + P^T}{2}$. (This is not an integer matrix, but still gives an integer quadratic form.)

We then obtain the following corollary.
\begin{cor}
Two $2 \times 2$ Seifert matrices $P_1$ and $P_2$ with the same Alexander polynomial $\Delta_m$ (equivalently, with the same determinant $m$) are $S$-equivalent if and only if there exists $X\in \SL_2(\Z[\frac{1}{m}])$ with $P_1 = X P_2 X^T$. 
\end{cor}
This is a special case of 4.15 in Trotter \cite{trotter-s-equivalence}.  For larger Seifert matrices Trotter also shows that $S$-equivalence implies $\Sp_{2g}(\Z[\frac{1}{m}])$-equivalence, but the converse is not generally the case.

Another corollary (which can also be proved directly):
\begin{cor}\label{part1}
Any binary quadratic form over $\Z[\frac{1}{m}]$ of discriminant $1-4m$ is $\SL_2(\Z[\frac{1}{m}])$-equivalent to a form defined over $\Z$.
\end{cor}

\subsection{The map from oriented ideal classes of $\Z[\gamma_m])$ to oriented ideal classes of $R_m$}

Corollary~\ref{part1} can also be interpreted as a statement about oriented class groups.  Let $\gamma_m = \frac{1 + \sqrt{1-4m}}{2}$, so that $\Z[\gamma_m]$ is the ring of integers in $\Q(\sqrt{1-4m})$.

The inclusion $\iota: \Z[\gamma_m] \into R_m$ induces a map $\iota_*: \Cl^+(\Z[\gamma_m]) \to \Cl^+(R_m)$.  More generally, if $(I, \kappa)$ is a (possibly non-invertible) ideal class of $\Z[\gamma_m]$, then we can map it to the ideal class $(I R_m, \kappa)$ of $R_m$.  In this language, Corollary~\ref{part1} becomes:
\begin{cor}\label{cor: surjective}
The map $\iota_*: \Cl^+(\Z[\gamma_m]) \to \Cl^+(R_m)$ is surjective.
\end{cor}

The kernel of $\iota_*$ can also be described explicitly: 
\begin{prop}\label{describe kernel}
The kernel $\ker \iota_*$ is generated by the classes \[
\mathfrak{g}_p = [(p, \gamma_m), p][(p, 1-\gamma_m), p]^{-1}=[(p, \gamma_m), p]^2
\]
where $p$ runs through the set of all prime factors of $m$.
\end{prop}
\begin{proof}
First, observe that $R_m = \Z[\gamma_m, \frac{1}{m}]$, and that $m$ factors in $R_m$ as 
\[
m = \prod_{p \mid m} (p, \gamma_m)(p, 1 - \gamma_m).
\]

Let $[I, \kappa]$ be an arbitrary element of $\ker \iota_*$; since we are in the kernel we can rescale so that $I \cdot R_m = R_m$ and $\kappa =1$.  Then $I$ is a fractional ideal of $\Z[\gamma_m]$ which becomes trivial when we invert the element $m$, so we can factor $I$ as 
\[
I =  \prod_{p \mid m}(p, \gamma_m)^{a_p}(p, 1 - \gamma_m)^{b_p}.
\]
In order to have $N I = (1)$ we must have $a_p = -b_p$ for all $p$.  The result follows.
\end{proof}
Note that not all oriented ideal classes of $\Z[\gamma_m]$ are invertible.  However:
\begin{prop}\label{kernel}
Two (possibly non-invertible) oriented ideal classes $[I, \kappa]$ and $[I', \kappa']$ of $\Z[\gamma_m]$ become equivalent in $R_m$ if and only if there exists $[J, \lambda] \in \ker \iota_*$ with \[
[I, \kappa] = [J, \lambda][I', \kappa'].
\]
\end{prop}
\begin{proof}
As before, we can reduce to the case $\kappa = \kappa'$.  For each $p$ dividing $m$ let \[
a_p = v_{\p}(I) - v_{\p}(I')
\]
where $\p$ is the invertible ideal $(p, \gamma_m)$.  Let
\[J = \prod_{p \mid m} [(p, \gamma_m), p]^{a_p}[(p, 1-\gamma_m), p]^{-a_p}.\]

One can then check locally that $I = J I'$.
\end{proof}

\begin{cor}\label{cor: kernel is squares}
We have $\ker \iota_* \subset \Cl^+(R_m)^2$; that is, the kernel is contained in the principal genus.
\end{cor}
This corollary can also be proved from the point of view of quadratic forms: two quadratic forms in different genera are not equivalent over $\Q$, so cannot be equivalent over $\Z[1/m]$ either.

\subsection{Consequences for prime $p$}\label{consequences}

In the case when $m = \pm p$ is prime, we have as a consequence that
\begin{prop}\label{bijection}
If $m$ is prime (possibly negative) then the map $\iota_*: \Cl^+(\Z[\gamma_m]) \to \Cl^+(R_m)$ is an isomorphism.  More generally, oriented ideal classes of $\Z[\gamma_m]$ are in bijection with oriented ideal classes of $R_m$ via $[I, \kappa] \mapsto [I R_m, \kappa]$.
\end{prop}

Translating back into the language of quadratic forms, we also have the important corollary:
\begin{cor}\label{part2}
Two quadratic forms over $\Z$ of determinant $1-4p$ are $\SL_2(\Z[\frac{1}{p}])$-equivalent if and only if they are $\SL_2(\Z)$-equivalent.
\end{cor}
This corollary is also implied by Trotter's work in \cite{trotter-s-equivalence}.  

We now prove Theorem~\ref{prime lower}, which shows that Heuristic~\ref{prime heuristic} is at least correct up to a a factor of $X^{\epsilon}$:

\begin{proof}[Proof of Theorem~\ref{prime lower}]

We show the equivalent statement, that the total is $\gg X^{3/2-\epsilon}/\log X$ for any $\epsilon > 0$.

Choose any $\epsilon > 0$.  By the Brauer-Siegel theorem plus the formula for class number of non-maximal orders, there exists some constant $c_\epsilon$ such that the size 
\[
|\Cl^+(\Z[\gamma_m])| \ge c_\epsilon X^{1/2-\epsilon} 
\]
for every $m \in [0, X]$.  We have just seen that $ |\Cl^+(\Z[\gamma_m])| = |\Cl^+(R_m)|$ when $m$ is prime.  Hence the total contribution of the  $\sim X/ \log X$ primes in $[0, X]$ is $\gg X^{3/2-\epsilon}/\log X$.
\end{proof}

%Note that these ideals classes satisfy one relation, coming from the identity
%\[
%\prod_{p \mid m} (p, \gamma_m)^{v_p(m)} = (\gamma_m).
%\]

\section{Heuristics}\label{heuristics}

\subsection{Derivation of Heuristics~\ref{prime heuristic} and \ref{negative prime heuristic}}
We now justify Heuristics~\ref{prime heuristic} and \ref{negative prime heuristic} using the results of the previous section.

We first consider Heuristic~\ref{prime heuristic}: we wish to estimate the total number of distinct Alexander modules (with Blanchfield pairing) having Alexander polynomial equal to $\Delta_p$ for some prime $p$ in the range $[1, X]$.  By Theorem~\ref{classification}, these are in bijection with oriented ideal classes of $R_m$.  Note that these ideal classes may not be invertible: however we expect that restricting to invertible ideal classes will not change the asymptotic order of growth (a positive proportion of $m$ have $1-4m$ squarefree).   Therefore the order of growth we seek should be the same as that of

\[
\sum_{m \in (0, X] \text { prime}} |\Cl^+(R_m)| = \sum_{m \in (0, X] \text { prime}} |\Cl^+(\Z[\gamma_m])|,
\]
where the equality is by Corollary~\ref{bijection}.

  We know by Gauss that $\sum_{m \in (0, X]} |\Cl^+(\Z[\gamma_m])| \asymp X^{3/2}$, where the sum is over all $m$.  If we then assume restricting to $m$ prime does not change the distribution of class numbers, we obtain that \[
  \sum_{m \in (0, X] \text{ prime}} |\Cl^+(\Z[\gamma_m])| \asymp X^{3/2}/(\log X),\] as in Heuristic~\ref{prime heuristic}.

We now wish to do the same thing for Heuristic~\ref{negative prime heuristic}, where $m$ ranges over negative primes.   In this case, we don't know the exact asymptotics of $\sum_{m \in [-X, 0)} |\Cl^+(\Z[\gamma_m])|$.  However Hooley \cite{hooley-quadratic} has a heuristic on average class numbers of real quadratic fields, which Cohen and Lenstra \cite{cohen-lenstra} note is compatible with their heuristics.  It predicts that $\sum_{m \in [-X, 0)} |\Cl^+(\Z[\gamma_m])| \asymp X \log^2 X$.  (Technically, Hooley only considers non-square discrimiants, but one can calculate the class number directly in the case of $1-4m = a^2$, and the total contribution is $O(X)$). Hence, after restricting to $m$ prime, we expect \[
\sum_{\substack{m \in [-X, 0) \\ |m| \text { prime}}} |\Cl^+(\Z[\gamma_m])| \asymp X \log X,
\]
as in Heuristic~\ref{negative prime heuristic}.  

\begin{rem}
Empirical evidence suggests that the heuristic reasoning above is an oversimplification: the class numbers $h(1-4p)$ are on average smaller than arbitrary $h(1-4m)$.  This appears to be related to congruence restrictions on the numbers $1-4p$: for any modulus $a$, we can only have $1-4p \equiv 1 \pmod a$ for at most one value of $a$.  Hence the heuristic arguments above give the wrong constants, but I conjecture that they at least give the right order of growth.
\end{rem} 

In the case where $m = \pm p^e$, a similar analysis applies.  The surjection $\iota_*: \Cl^+(\Z[\gamma_m]) \to \Cl^+(R_m)$ is no longer a bijection, but it does have kernel of size bounded by $e$, so our expected average size of the class group will have the same order of growth.  However, in this case we will obtain a smaller total, since the number of $m \in (0, X]$ that are prime powers is only $o(X^{1/2})$.

\subsection{Background on Cohen-Lenstra heuristics}

We now move on the the case of general $m$:  since $\iota_*$ is no longer injective, we will need to have information on the expected structure of the group $\Cl^+(\Z[\gamma_m])$, not only its size.  Fortunately, such information is provided by the Cohen-Lenstra heuristics.

In their landmark paper \cite{cohen-lenstra}, Cohen and Lenstra defined a sequence of distributions $\{\mu^u\}_{u \ge 0}$ on the set of finite abelian groups, where $\mu^0$ was expected to model the distribution of class groups of imaginary quadratic fields and $\mu^1$ the distribution of class groups of real quadratic fields.  (More generally, $\mu^i$ is supposed to model the distribution of class groups of fields with $i+1$ infinite places.)  We give a rough definition here:

\begin{defn} For a non-negative integer $u$, the Cohen-Lenstra distribution $\mu^u$ gives a finite abelian group $G$ weight proportional to $\frac{1}{|\Aut(G)| |G|^u}$.
\end{defn}

For $u \ge 1$, the sum of all weights converges to $1$, so $\mu^u$ gives a measure on the set of all finite abelian groups, where any given $G$ appears with positive probability.  Cohen and Lenstra showed (\cite{cohen-lenstra}, Thm 5.8) that for $u \ge 2$, the average size of $G$ is finite, and goes to $1$ as $u \to \infty$, while for $u = 1$ the average size of $G$ is unbounded.

On the other hand, for $u = 0$ the sum of all weights diverges, reflecting the fact that for imaginary quadratic fields we see any given class group only finitely many times.  However, for many ``reasonable'' functions $f$ on finite abelian groups, Cohen and Lenstra were still able to make sense of the average value $\mathbb{E}(f) = \text{``$\int f d\mu^0$''}$ of $f$ evaluated at a random group drawn from $\mu^0$, by ordering the set of finite abelian groups in a natural way (e.g. by size) and taking a limit.  For instance, one can explicitly calculate the probability that a random group drawn from $\mu^0$ is cyclic; the answer is positive but less than $1$.

Cohen and Lenstra noticed that their family of  distributions $\mu^u$ were related by the following:
\begin{prop}[Cohen-Lenstra \cite{cohen-lenstra}]\label{relation}
Let $u, k \ge 0$.  Suppose $G$ is randomly selected from $\mu^{u}$, and $g_1, \dotsc, g_k$ are randomly selected elements of $G$.  The distribution of $G / \langle g_1, \dotsc, g_k \rangle$ agrees with $\mu^{u+k}$. 
\end{prop}

The numerical evidence on class groups of quadratic fields indicates that the Cohen-Lenstra heuristics correctly predict the $p$-part of the class group for $p$ odd.  On the other hand, Cohen and Lenstra noticed that their heuristics could not be right for the 2-part of the class group, as they conflicted with genus theory.  Later, Gerth \cite{gerth_4, gerth_general} observed that this was easily fixed by restricting to the principal genus, that is, the squares in the oriented class group.

\begin{heuristic*}[Cohen-Lenstra, refined by Gerth]
The distribution of $\Cl^+(\cO_K)^2$ as $K$ runs through all imaginary quadratic fields agrees with $\mu^0$.

The distribution of $\Cl^+(\cO_K)^2$ as $K$ runs through all real quadratic fields agrees with $\mu^1$.
\end{heuristic*}

%Cohen and Lenstra also observed that the the numerical data also suggests that class groups of families of rings of the form $\cO_K[\frac{1}{p}]$, where $p$ is a prime that splits in $K$, show similar behavior to those of real quadratic fields, that is, should also be modeled by the distribution $\mu_1$.  More generally, 

%\begin{heuristic*}[Cohen-Lenstra for $\cO_{K, S}$]
%Suppose $K$ is a randomly chosen quadratic field and $S$ is a randomly chosen set of places of $K$ subject to $|S|=n$ and $S$ contains all the infinite places. Then the distribution of $\Cl^+(\cO_{K, S})^2$ agrees with $\mu^n$. 
%\end{heuristic*}

%This heuristic can in fact be deduced from the original Cohen-Lenstra heuristic using Proposition~\ref{relation} and the Chebotarev density theorem.  However, it can only be applied to families $\cO_{K, S}$ where $S$ is chosen independently from $K$, so, as we will see, we cannot apply it directly to the rings $R_m$.

Cohen and Lenstra noted that their heuristic would need to be modified in the case of quadratic orders, based on what we know about the relationship between $\Cl^+(\cO_D)$ and $\Cl^+(\cO_K)$ for a non-maximal order $\cO_D \subset \cO_K$.  As far as we know, the details of this modification have not been written down in full, though Bhargava and Varma \cite{bhargava-varma} have proved some results confirming the Cohen-Lenstra heuristics for $3$-torsion in quadratic orders.

\subsection{Heuristics for our problem}
For the remainder of the discussion, we will ignore these issues with non-maximal orders, and restrict to ourselves to considering heuristics for the class groups $\Cl^+(R_m)$ for those $m$ with $1-4m$ squarefree. 

There is one bit of information about the class group $\Cl^+(R_m)$ that we can determine exactly: we can obtain the $2$-rank by genus theory.  We first note that by Corollaries~\ref{kernel} and \ref{cor: kernel is squares} we know that  \[ \Cl^+(R_m)/ \Cl^+(R_m)^2 \isom \Cl^+(\Z[\gamma_m]) / \Cl^+(\Z[\gamma_m])^2\] where the rank $r_m$ equals the number of distinct prime factors $\omega(m)$ of $m$, plus a bounded correction term.  By Corollaries~\ref{kernel} and \ref{cor: kernel is squares}, that $\Cl^+(\Z[\gamma_m])$ also has the same rank $r_m$.

As previously noted, genus theory says that the $2$-rank of the oriented class group $\Cl^+(R_m)$ equals $\omega(m)$ plus a bounded correction term.  We will make the heuristic assumption that $\Cl^+(R_m) / \Cl^+(R_m)^2$ and $\Cl^+(R_m)^2$ are independent of each other.  It then remains to give heuristics for $\Cl^+(R_m)^2$.

By Corollaries~\ref{cor: surjective} and \ref{cor: kernel is squares}, we have a short exact sequence
\[
1 \to \ker i_* \to \Cl^+(\Z[\gamma_m])^2 \to \Cl^+(R_m)^2 \to 1.
\]

Because we assumed $1-4m$ squarefree, $\Z[\gamma_m]$ is the full ring of integers of $\Q(\sqrt{1-4m})$.   The Cohen-Lenstra heuristics predict that the distribution of the groups $\Cl^+(\Z[\gamma_m]^2)$ is modeled by $\mu_0$ as $m$ ranges over the positive integers and by $\mu_1$ as $m$ ranges over the negative integers.

Furthermore, Proposition~\ref{describe kernel} tells us that $\ker i_*$ is generated by the classes $\mathfrak{g}_p = [(p, \gamma_i), p]^2$ for $p$ dividing $m$. 

We know one relation between these classes, namely that if $m = \prod_{i = 1}^k p_i^{n_i}$, then 
\begin{equation*}
\begin{split}
\prod_{p \mid m}\mathfrak{g}_{p}^{v_p(m)} &=  \prod_{p \mid m}  [(p, \gamma_m), p]^{v_p(m)} \prod_{p \mid m} [(p, 1-\gamma_m),p]^{-v_p(m)}\\
&= [(\gamma_m), m][((1-\gamma_m)^{-1}), m]^{-1} \\
&= [(\gamma_m) (1- \gamma_m)^{-1}, 1] 
\end{split}
\end{equation*}
which is the trivial class.

This motivates the following heuristic:
\begin{heuristic}
Let $n_1, \dotsc, n_k$ be fixed positive integers.  Let $m$ run through all positive integers of the form $p_1^{n_1} \dotsb p_k^{n_k}$ with $1-4m$ squarefree.  The distribution of the finite abelian groups $(\Cl^+(R_m))^2$ agrees with the distribution of $G/\langle g_1, \dotsc, g_k \rangle$, where $G$ is a finite abelian group selected from the Cohen-Lenstra distribution $\mu_0$ and $g_1, \dotsc g_k$ are randomly chosen elements of $G$ subject to the constraint that $\prod {g_k}^{n_k} = 1$.
\end{heuristic}
\begin{heuristic}
Let $n_1, \dotsc, n_k$ be fixed positive integers.  Let $m$ run through all negative integers of the form $-p_1^{n_1} \dotsb p_k^{n_k}$ with $1-4m$ squarefree.  The distribution of the finite abelian groups $\Cl^+(R_m)$ agrees with the distribution of $G/\langle g_1, \dotsc, g_k \rangle$, where $G$ is a finite abelian group selected from the Cohen-Lenstra distribution $\mu_1$ and $g_1, \dotsc g_k$ are randomly chosen elements of $G$ subject to the constraint that  $\prod {g_k}^{n_k} = 1$.
\end{heuristic}

If some $n_i = 1$, then $g_i = \prod_{k \ne i} g_k^{-n_k}$ is uniquely determined by the other $g_k$, which are now independently chosen random elements of $G$.  In these cases, the heuristics can be simplified to
\begin{heuristic}\label{positive not powerful}
Let $n_1, \dotsc, n_k$ be fixed positive integers such that some $n_i = 1$.  Let $m$ run through all positive integers of the form $p_1^{n_1} \dotsb p_k^{n_k}$ with $1-4m$ squarefree.  The distribution of the finite abelian groups $(\Cl^+(R_m))^2$ agrees with $\mu_{k-1}$.
\end{heuristic}
\begin{heuristic}\label{negative not powerful}
Let $m$ run through all negative integers of the form $-p_1^{n_1} \dotsb p_k^{n_k}$ with $1-4m$ squarefree.  The distribution of the finite abelian groups $(\Cl^+(R_m))^2$ agrees with $\mu_{k}$.
\end{heuristic}
Note that these heuristics apply to all integers $m$ which are not powerful.   It is known the number of integers $m \in [0, X)$ which are powerful is $O(X^{1/2})$ for some constant $c$.  As a result it will generally be safe to ignore them when obtaining heuristics for the total behavior.

These heuristics are fairly na\"{i}ve and it is worth investigating them further for accuracy, but I conjecture that they at least can be used to predict the correct order of magnitude for the average sizes of these groups.  

\begin{rem}
We have not quite used all the information that we can obtain about $\ker \iota_* = \langle \mathfrak{g}_p: p \mid m \rangle$.  As we will see in the proof of Lemma~\ref{key lemma}, in the case of $m$ positive (so negative discriminant $1-4m$), the kernel $\ker \iota_*$ must have size equal to at least the number $d_{\le m^{1/4}}(m)$ of factors of $m$ that are at most $m^{1/4}$.  However, this is compatible with our heuristic:  $\Cl^+(R_m)^2$ is predicted by Cohen-Lenstra to have small non-cyclic part, but also has size $\gg m^{1/2 - \epsilon}$ Hence we already expect that $\mathfrak{g}_1, \dotsc, \mathfrak{g}_k$ will generate a large subgroup of $\Cl^+(R_m)^2$ with high probability.
\end{rem}

\subsection{Totals over families of $m$ with a fixed number of prime factors}
 We first look at some special families of $m$ with a small number of prime factors.  

First of all, in the cases $m = p$ and $m  = -p$, where $\ker \iota_*$ is trivial and $\Cl^+(R_m)^2 \isom \Cl^+(\Z[\gamma])^2$, we recover Heuristic~\ref{prime heuristic} and Heuristic~\ref{negative prime heuristic}.  In the closely related, and previously mentioned, case where $p^e$ is a positive prime power, we obtain heuristics with a similar average rate of growth, but since there are only $o(X^{1/2})$ prime powers up to $X$, we expect the total will only grow as $o(X)$.  For the remainder of this section, we will only consider non-powerful integers $m$, as the powerful values of $m$ are sparse and will only contribute error terms.

We now consider the special case of the family $m = + p_1 p_2$:  here we predict that $\Cl^+(R_m)^2$ is modeled by the Cohen-Lenstra distribution $\mu_1$.  As previously noted, the expected size of a randomly selected group drawn from $\mu_1$ is infinite, so we need an additional heuristic here for how fast the average size over discriminants $\le X$ goes to infinity as $X \to \infty$.

We argue by analogy with Hooley's results \cite{hooley-quadratic} on the distribution of sizes of class groups of real quadratic fields, in which case the principal genus is also modeled by the same distribution $\mu_1$.  Hooley's heuristics imply that over all real quadratic fields of discriminant $\le X$, the average size of the class group grows like a constant times $\log^2 X$, and the average size of the principal genus grows like a constant times $\log X$.  (The latter heuristic is only formulated for prime discriminants, where the size of the principal genus is fixed, but we expect it to apply to other families also.)

Hence, in our setting, we conjecture that the average size of $\Cl^+(R_m)^2$ over all 
$m = p_1 p_2 \in (0, X]$ is asymptotically a constant times $\log X$.  We also know that $\Cl^+(R_m)/\Cl^+(R_m)^2$ has order equal to $2^{\omega(1-4m)+O(1)}$, which, when averaged over all $m = p_1 p_2 \in (0, X]$, is asymptotic to (a different) constant times $\log X$. 

By our assumption of independence, we deduce that heuristically the average size of $\Cl^+(R_m)$ over all $m = p_1 p_2 \in (0, X]$ grows like $(\log X)^2$.  Now the total number of $m \in (0, X]$ that are a product of two primes grows like a constant times $X/(\log X)$, and so we obtain the heuristic
\begin{equation} \label{eq:two primes}
    \sum_{m = p_1 p_2 \in (0, X]} |\Cl^+(R_m)| \asymp X \log X.
\end{equation}

We now consider the case of $m$ positive having $k \ge 3$ distinct prime factors.  Here we expect (by Heuristic~\ref{positive not powerful}) that $\Cl^+(R_m)^2$ is a group randomly drawn from the distribution $\mu_{k-1}$.  Since $k -1 \ge 2$, we know that the average size of a group drawn from $\mu_{k-1}$ is finite.  Hence, heuristically, the average of $|\Cl^+(R_m)^2|$ over all $m \in (0, X]$ with $k$ prime factors should tend to a constant as $m \to \infty$.  Taking into account the principal genus, which as before has average size $\asymp \log X$, we obtain  the following asymptotic for the average:

\begin{equation} \label{eq: k primes}
   \frac{ \sum_{m = p_1 p_2 \dotsm p_k \in (0, X]} |\Cl^+(R_m)| } { \sum_{m = p_1 p_2 \dotsm p_k \in (0, X]} 1 }\sim c^{+}_k \log X.
\end{equation}
where the constant $c^{+}_k$ goes to $1$ as $k \to \infty$, and in particular is bounded.

By the same argument for $m$ negative with $k \ge 2$ distinct prime factors, we obtain

\begin{equation} \label{eq: minus k primes}
   \frac{ \sum_{m = p_1 p_2 \dotsm p_k \in (0, X]} |\Cl^+(R_m)| } { \sum_{m = p_1 p_2 \dotsm p_k \in (0, X]} 1 }\sim c^{-}_k \log X
\end{equation}
where the constant $c^{-}_k$ also goes to $1$ as $k \to \infty$.

\subsection{Justification of Heuristic~\ref{not prime}} 

We are now ready to justify Heuristic~\ref{not prime}.  We break up the sum we need to bound by the sign of $m$ and the number of distinct prime factors $\omega(m)$ of $m$.
\begin{multline*}
\sum_{\substack{m \in [-X, X] \\ |m| \text{ not prime}}} |\Cl^+(R_m)| = \sum_{\substack{m =  \pm p^k \in [-X, X]  \\ k > 1}} |\Cl^+(R_m)| +  \sum_{\substack{m  \in (0, X] \\ \omega(m) = 2}} |\Cl^+(R_m)| \\ +  \sum_{\substack{m  \in (0, X] \\ \omega(m) \ge 3}} |\Cl^+(R_m)|  + \sum_{\substack{m \in [-X, 0) \\ \omega(m) \ge 2}}  |\Cl^+(R_m)|.
\end{multline*}

The first term, $\sum_{\substack{m =  \pm p^k \in [-X, X]  \\ k > 1}} |\Cl^+(R_m)|$, we expect is $o(X)$ by the discussion above.  However, we will see that all the remaining three terms contribute $\asymp X \log X$

First of all, we expect $\sum_{\substack{m  \in (0, X] \\ \omega(m) = 2}} |\Cl^+(R_m)| \sim \sum_{\substack{m  \in (0, X] \\ m = pq}} |\Cl^+(R_m)| $ since a density $1$ subset of the $m$ with two prime factors are of the form $m = pq$, so this total is heuristically $\asymp X \log (X)$ by \eqref{eq:two primes}.

The other two terms give a heuristic total of $X \log X$ for a different reason.  The second sum \[\sum_{\substack{m  \in (0, X] \\ \omega(m) \ge 3}} |\Cl^+(R_m)| \] is a sum of $\asymp X$ terms which are heuristically on average $\asymp \log X$ by \eqref{eq: k primes}, and the same for 
\[
\sum_{\substack{m \in [-X, 0) \\ \omega(m) \ge 2}}  |\Cl^+(R_m)|
\]
by \eqref{eq: minus k primes}.

Hence, our total sum is $\asymp X \log (X)$ as in Heuristic 3, and both the subset of $m$ having the form $+p_1 p_2$ and the complement of this subset contribute main terms.

\section{Proof of Theorem~\ref{prime upper} by Sieving}
By Corollaries~\ref{part1} and \ref{part2}, it's enough to show 

\begin{prop}\label{sl2form}
The total number of $\SL_2(\Z)$-equivalence classes of binary quadratic forms $ax^2 + bxy + c y^2$ with prime discriminant of the form $p = 1-4m$ for $m \in [0, X]$ is bounded above by \textit{$O(\frac{X^{3/2}}{\log X})$.}
\end{prop}

We will actually show that the total for $m \in [X, 2X]$ is also $O(\frac{X^{3/2}}{\log X})$, and the proposition will follow by summing.  Also, we will only count the positive definite quadratic forms, as the count of negative definite forms is the same.

We follow the approach of Rosser's sieve   \cite{iwaniec-rossers_sieve}, modifying the terminology to suit our approach.  We introduce an auxiliary parameter $ Z \le X$ whose value will be chosen later, and let $P(Z)$ denote the product of all primes up to $Z$.  Let
\[
\mathcal{F} = \{(\alpha, \beta, \gamma) \in \R^3 \mid |\beta| \le \gamma \le \alpha\}
\]
be the standard fundamental domain for $\SL_2(\Z)$ acting on positive definite binary quadratic forms.

Then the total we wish to bound is at most:
\[
S(X, Z) := \sum_{\substack{X \le m \le 2X \\ (m, P(Z))=1}}  \#\{(a, b, c) \in \Z^3 \cap \mathcal{F}  : b^2 -4ac =1-4m\}
\]
Note here that $b^2 -4ac =1-4m$ implies $b$ odd: we write $b = 2b' + 1$ and let $\mathcal{F}'$ be the preimage of $\mathcal{F}$ under the affine transformation $(\alpha, \beta, \gamma) \mapsto (\alpha, 2\beta+1, \gamma)$.  Using this change of variables
\[
S(X, Z) = \sum_{\substack{X \le m \le 2X \\ (m, P(Z))=1}}  \#\{(a, b', c) \in \Z^3 \cap \mathcal{F}'  : ac - b'(b'+1) = m \}
\]

To apply the sieve, we need estimates on the following quantities for all squarefree $d \le X$:
\begin{equation}
S_d(X) := \sum_{\substack{m \in [X, 2X] \\ d \mid m}} \#\{(a, b', c) \in \Z^3  \cap \mathcal{F}' : ac - b'(b'+1)=m\}.
\end{equation}
\begin{lemma}
For a positive integer $s$,  let $\rho(s) = \prod_{p \mid s} \frac{p+1}{p^2}$.  If $d$ is a squarefree positive integer $\le X$, there exist explicit real constants $c_1, c_2$ such that
\begin{equation}
S_d(X) = c_1 \rho(d) X^{3/2} + R_d(X),
\end{equation}
 and the error term $R_d(X)$ is bounded by
\begin{equation}\label{bound R_d}
|R_d(X)| \le c_2 X \rho(d) (\max (d, \log X)).
\end{equation}
\end{lemma}
\begin{proof}
We observe that for all squarefree $d$, $S_d(X)$ counts the number of points in the intersection of the region 
\[\mathcal{R}_X = \{(\alpha, \beta', \gamma) \in \mathcal{F}' : | \alpha \gamma - \beta'(\beta'+1)| \le X\] 
with the union of the cosets of $(d \Z)^3$ on which the function $ac - b(b+1)$ vanishes modulo $d$.

We wish to apply:
\begin{lemma}[Davenport] \cite{davenport}
Let $R$ be a bounded semi-algebraic region in $\R^n$, defined by $k$ polynomial inequalities of degree at most $\ell$.  Then the number of points $(a, b, c)\in \Z^n \cap R$ can be asymptotically expressed as 
\[
\vol(R) + \epsilon(R)
\]
with the error term $\epsilon(R)$ bounded in size by $\epsilon(R) < \kappa \max(\vol (\overline{R}), 1)$ where $\overline{R}$ runs over all projections of $R$ onto subspaces of $\R^n$ spanned by coordinate axes, and $\kappa = \kappa(n, m, k, \ell)$ is some explicit constant depending only on $n, m, k$, and $\ell$.
\end{lemma}

We cannot apply Davenport's lemma directly because $\mathcal{R}_{X}$ goes off to infinity.  Instead, we truncate the cusp: for a positive real parameter $R$, define
\[
\mathcal{R}_{X, T} = \{(x, y, z)\in \mathcal{R}_X \mid z < T\}.
\]
We observe that any lattice point $(a, b, c) \in \mathcal{R}_X$ has $c \le 2X$, so also belongs to $R_{X, 2X}$.

One can calculate that the largest $1$-dimensional projection of $R_{X, 2X}$ has length $2X$, while the largest $2$-dimensional projection of $R_{X, 2X}$ has area $c_3 X \log X$ for an explicit constant $c_3$.

Now let $L_1, \dotsc, L_n$ be the cosets of $(d \Z)^3$ for which $ac - b'(b'+1) \equiv 0 \pmod d$ for all $(a, b', c) \in L_i$.  The number $n$ is equal to the number of solutions to $ac - b'(b'+1) = 0$ in $(\Z/d\Z)^3$.   A calculation with the Chinese remainder theorem gives  $n= \rho(d) d^3$.

Applying Davenport's lemma to $R_{X, t}$ rescaled by $d^{-1}$ and translated appropriately, we obtain that there exists a real number $\kappa$ such that for each $i$
\begin{equation}\label{apply davenport}
\begin{split}
\# (R_{X,2X} \cap L_i) - c_1 d^{-3} X^{3/2} &\le \kappa(\max(d^{-2}X \log X, d^{-1} X, 1)\\
&= \kappa d^{-2} X \max(\log X, d)  
\end{split}
\end{equation}
where the last step uses $d \le X$.

Summing over all $\rho(d)d^3$ values of $i$ and applying the triangle inequality, we obtain
\begin{equation*}
\begin{split}
S_d(X) - \rho(d) c_1 d^{-3} X^{3/2} &= \sum_{i} (\# (R_{X,2X} \cap L_i) - c_1 d^{-3} X^{3/2})\\
& \le \rho(d) \kappa d^{-2} X \max(\log X, d) 
\end{split}
\end{equation*}
as desired.

\end{proof}

We are now ready to prove Proposition~\ref{sl2form}.

\begin{proof}
We apply Rosser's sieve.   First we calculate the ``sieving density,'' also known as the ``dimension''.  The following inequality is analogous to (1.3) in \cite{iwaniec-rossers_sieve}: for all $Z>W \ge 2$ we have

\begin{equation}
\prod_{W \le p < Z} ( 1 - \rho(p))^{-1} \le \left(\frac{\log Z}{\log V} \right)^{\kappa} \left(1 + \frac{K}{\log W} \right).
\end{equation}
where $\kappa = 1$ and $K$ is a sufficiently large constant.  This is true by comparing to the product $\prod_{W<p<Z} (1 - 1/p)$ and applying Mertens' formula for the asymptotic growth of the latter.

Therefore we may apply (the first half of) Theorem 1.4 of \cite{iwaniec-rossers_sieve} with $y= Z$ (so that $s=1$) to obtain
\begin{equation}\label{using rosser}
S(X, Z) < X^{3/2} \prod_{p<Z} (1- \rho(p)) \left(F(1) + e^{\sqrt{K}} Q(1) (\log Z)^{-1/3} \right) + \sum_{d < Z \text { squarefree}} |R_d(X)|
\end{equation}
where $F(s)$, $Q(s)$ are specific functions defined in \cite{iwaniec-rossers_sieve}; we will not need any properties of them, just that $F(1)$ and $Q(1)$ are explicit constants.

Using our previous result that $\prod_{p<Z} (1 - \rho(p)) \asymp 1/\log(Z)$, we see that the first term is $O(X^{3/2}/\log Z)$.

Applying \eqref{bound R_d} to the second term gives

\begin{equation} \label{summing bounds}
\sum_{d < Z \text { squarefree}} |R_d(X)| \le c_1 X \sum_{d < Z \text{ squarefree}} \rho(d) d (\max (d, \log X)).
\end{equation}

We estimate $\rho(d)$.  Since $d$ is squarefree, we can write $d$ as a product of distinct primes: $d = \prod_{i = 1}^{n_d} p_i$.  Then we make the crude bound
\begin{equation}\label{bounding rho}
\begin{split}
\rho(d) &= \prod_{i = 1}^{n_d} \frac{p_i + 1}{p_i^2} \\
& = \frac{1}{d} \prod_{i =1}^{n_d} \left(1  + \frac{1}{p_i} \right)\\
& = \frac{1}{d} \sum_{d' \mid d} \frac{1}{d'}\\
& \le \frac{1}{d} \sum_{1  \le d' \le Z} \frac{1}{d'}\\
& \le \frac{\log Z + 1}{d}.
\end{split}
\end{equation}

Plugging \eqref{bounding rho} into the sum in \eqref{summing bounds}, we obtain
\begin{equation}\begin{split}
\sum_{d < Z \text { squarefree}} |R_d(X)| &\le c_1 (\log Z + 1)  \sum_{d < Z} \max(d, \log x)\\
& \le c_1 (\log Z + 1) \sum_{d<Z} (d+\log X) \\
& \le c_1 X (\log Z + 1)(Z^2 + Z \log X) \\
& \le c_1 X (\alpha \log X + 1)(X^{2\alpha} + X^{\alpha} \log X).
\end{split}
\end{equation}
In the last line we have set $z = X^\alpha$.

We deduce the following asymptotic for our error term, where we have fixed  $\alpha$ and allow $X$ to vary:

\begin{equation}
\sum_{d < z \text { squarefree}} |R_d(X, z)| = O(X^{1 + 2\alpha} \log X).
\end{equation}

This will be $o(X^{3/2}/\log X)$ for any $\alpha < 1/4$.  

We conclude that if we set $Z= X^{\alpha}$ for a fixed $\alpha <1/4$, the main term in \eqref{using rosser} is $O(X^{3/2}/\log X)$ and the error term is $o(X^{3/2}/\log X)$, giving the desired asymptotic.
\end{proof}

\section{Proof of Theorem~\ref{aggregate}}
We now prove Theorem~\ref{aggregate}.  We will prove it in the equivalent form

\begin{thm}\label{aggregate quad forms}
If $T(X)$ is the number of $\SL_2(\Z[\frac{1}{m}])$-equivalence classes of binary quadratic forms of discriminant $1-4m$ as $m$ runs through all integers in the range $[-X, X]$, then
\[
\lim_{X \to \infty} \frac{T(X)}{X^{3/2}} = 0
\]
\end{thm}

Gauss's bound for the total number of $\SL_2(\Z)$-equivalence classes of binary quadratic forms with discriminant in this range is $O(X^{3/2})$, so this theorem says that strengthening the equivalence relation to $\SL_2(\Z[\frac{1}{m}])$-equivalence decreases the order of growth.

To prove this, we will need the following key lemma, which gives a lower bound on the size of the kernel of $\iota_* : \Cl^{+} (\Z[\gamma_m]) \to \Cl^+(R_m)$.  

\begin{defn}
	For an integer $m$, let $k_m = |\ker(\iota_*: \Cl^{+} (\Z[\gamma_m]) \to \Cl^+(R_m))|$.
\end{defn}

\begin{defn}
Let $\omega(m)$ denote the number of prime divisors of $m$.
\end{defn}

\begin{lemma}\label{key lemma}
    For any positive $m$, we have $k_m \ge \omega(m)-3$.
\end{lemma}  

\begin{proof}
    Write $r = m^{1/4}$.  We will show the stronger bound $k_m \ge \tau_{\le r} (m)$, where $\tau_{\le r}(m)$ is the the number of divisors of $m$ that are $\le r$.  This implies the lemma, as $m$ can have at most $3$ prime factors $> m^{1/4}$.

	The oriented ideal classes $\mathfrak{c}_s = [I_s, s]^2$, where $I_s = (s, \gamma_m)$ for $s$ a divisor of $m$, all belong to $\ker \iota_*$. 
	
	Hence it's enough to show that the classes $[I_s, s]^2$, for $1 \le s \le r$ dividing $m$, are all distinct.  We do this by showing that the fractional ideals $I_s^2= (s^2, \gamma_m-m)$ are not homothetic. 
	
	The element of smallest norm in $I_s^2$ is $s^2$, which has norm $s^4$ (here we are using $m$ positive and $> s^4$).  On the other hand, $I_s^2$ itself has norm $s^2$.  Since any homothety between two ideals must scale both the ideal norm and the norm of the smallest element by the same factor, this means that $I_s^2$ is not homothetic to $I_t^2$ for $s \ne t$.
\end{proof}

\begin{rem}
Note that we needed to use $m$ positive in this lemma.  Based on the Cohen-Lenstra heuristics for real quadratic fields, we expect $\Cl^+(R_m)^2$, so also its subgroup $\ker \iota_*$, to be trivial for a positive density of negative integers $m$.  Hence there should not be an analogous result for $m$ negative.
\end{rem}

We will need to be a little careful in our bookkeeping, because not all quadratic forms are primitive.  Recall that the {\em content} of a binary quadratic form $ax^2 + bxy + cy^2$ is defined to be $\content(Q) = \gcd(a, b, c)$ (so primitive forms have content $1$).  This notion still makes sense for binary quadratic forms over $\Z[1/m]$, where we normalize so that $\content(Q)$ is an integer relatively prime to $m$.  (Since all forms we work with have discriminant relatively prime to $m$, this will cause no confusion.)  In order to bound our weighted count we will first need to divide up by content.

\begin{defn} For $d \ge 1$, let be $T_d(X)$ the total number of $\SL_2(\Z[1/m])$-equivalence classes of binary quadratic forms with content $d$ and discriminant in $[-X, X]$.
\end{defn}

Dividing out by the content gives a bijection between $\SL_2(\Z[1/m])$-equivalence classes of  binary quadratic forms of discriminant $1-4m$ and content $d$ and $\SL_2(\Z[1/m])$-equivalence classes  of primitive binary quadratic forms of discriminant $(1-4m)/d^2$.   The latter in turn correspond to elements of the narrow class group $\Cl^+(\cO_{(1-4m)/d^2} [\frac{1}{m}])$.

We conclude that
\[
T_d(X) = \sum_{\substack{m \in [-X, X] \\ 4m \equiv 1 \pmod{ d^2}}} |\Cl^+(\cO_{(1-4m)/d^2} [\frac{1}{m}])|,
\]
where $\cO_{(1-4m)/d^2}$ is the quadratic order of discriminant $(1-4m)/d^2$.  We are now ready to prove Theorem~\ref{aggregate quad forms}.

\begin{proof}[Proof of Theorem~\ref{aggregate quad forms}]
We have $X^{-3/2} T(X) = \sum_{d \ge 1} X^{-3/2} T_d(X)$.  We claim that this sum satisfies the conditions of the dominated convergence theorem.  Indeed, for every $d$, $T_d(X)$ is at most the number of $\SL_2(\Z)$-equivalence classes of primitive binary quadratic forms with odd discriminant in the range $[-d^{-2}X, d^{-2} X]$. By Gauss's count of binary quadratic forms, the latter is bounded above by $c d^{-3} X^{3/2}$ for some uniform constant $c$.  Hence $ X^{-3/2} T_d(X) \le c d^{-3}$ for all $X$, and so the series  $\sum_{d \ge 1} X^{-3/2} T_d(X)$ is dominated by the convergent series  $\sum_{d \ge 1} c d^{-3}$.

So it suffices to prove $\lim_{X \to \infty} X^{-3/2} T_d(X) = 0$ for any fixed $d$.

We first reduce to the case $d=1$, by showing that the sum $T_1$ termwise dominates all the $T_d$. For this, note that the natural map $\Cl^+(\cO_{1-4m}) [\frac{1}{m}] \to \Cl^+(\cO_{(1-4m)/d^2)} [\frac{1}{m}])$ is surjective.  Hence 
\[
|\Cl^+(\cO_{1-4m}) [\frac{1}{m}]| \ge |\Cl^+(\cO_{(1-4m)/d^2}) [\frac{1}{m}]|
\]
as desired.

Hence it's enough to show that $T_1(X) = o(X^{3/2})$.  Write $T(X) = T_1^{\rm{pos}}(X) + T_1^{\rm{neg}}(X)$, where $T_1^{\rm{pos}}(X)$ is the total contribution of the terms with $m > 0$ and $T_1^{\rm{neg}}(X)$ is total of those with $m \le 0$.

We introduce the notation $h^+(D) = |\Cl^+(\cO_D)|$ for the oriented class number: for $D > 0$ this is the narrow class number, while for $D < 0$ it is twice the usual class number.

We can bound $T_1^{\rm{neg}}(X)$ by comparison with the partial sum of class numbers of positive discriminants.    Then
\[
T_1^{\rm{neg}}(X) \le  \sum_{m \in (-X, 0)} |h^+(1-4m)| \le \sum_{D \in (0, 4X) } h^+(D).
\]

It is well known that this latter sum is $o(X^{3/2})$; for completeness we give the argument here.  To dispose of the square values of $D$, note that $h(a^2) = O(a)$, so $\sum_{a \in (0, 2 \sqrt{X})} h^+(a^2) = O(X)$.  For the rest, we use Siegel's result \cite{siegel-average_measure}, proving a conjecture of Gauss, that $\sum_{D \ne \square \in (0, 4X)}  h^+(D)r(D) \sim c X^{3/2}$. Here the sum is over all non-square discriminants in the range $(0, 4X)$, and $r(D)$ is the regulator of the real quadratic order of discriminant $D$.  Since $r(D) \to \infty$ as $D \to \infty$, it follows that $\sum_{D \ne \square \in (0, 4X)}  h^+(D) = o(X^{3/2})$ also.

We must now show that $T_1^{\rm{pos}}(X) = o(X^{3/2})$.  

We have
\[
T_1^{\rm{pos}}(X) = \sum_{m \in (0, X)} |\Cl^+(\cO_{1-4m} [\frac{1}{m}])| =  \sum_{m \in (0, X)} \frac{h^+(1-4m)}{k_m}.
\]
by Lemma~\ref{key lemma}, where $h^+(1-4m)$ is the narrow class number of the quadratic order $\cO_{1-4m}$.  

Now, we choose a parameter $N$, and split the sum depending upon whether the number of prime divisors $\omega(m)$ is greater than or less than $N$:
\[
T_1^{\rm{pos}}(X) = T_1^\sharp(X) + T_1^\flat(X) = \sum_{\substack{m \in (0, X) \\ \tau(m) \ge N}} \frac{ h^+(1-4m)}{k_m} + \sum_{\substack{m \in (0, X) \\ \tau(m)<N}} \frac{2 h^+(1-4m)}{\tau(m)}
\]

The first term $T_1^\sharp(X)$ is easy to bound:
\begin{equation}\label{bound sharp}
T_1^\sharp(X) = \sum_{\substack{m \in (0, X) \\ \omega(m) \ge N}} \frac{h^+(1-4m)}{k_m} \le  \frac{1}{N}  \sum_{m \in (0, X)} h^+(1-4m) =  \kappa_1 \left ( \frac{X^{3/2}}{N} \right)
\end{equation}
for an explicit constant $\kappa_1$.

Now we deal with $T_1^{\flat}(X)$, which is the total contribution of those positive $m$ with $<N$ factors:  these have density $0$ for fixed $N$, so their contribution should be $o_N(X^{3/2})$.  We formalize this as follows: pick a real parameter $r$.  As before $\tau_r(m)$ denote the number of divisors of $m$ that are $< r$.  Then
\[
T_1^\flat(X) = \sum_{\substack{m \in (0, X]\\ \tau(m) < N}} \frac{2 h^+(1-4m)}{\tau(m)} \le 2 \sum_{\substack{m \in (0, X) \\ \tau_r(m) < N}} h^+(1-4m).
\]

We can then obtain an upper bound for this latter sum of class numbers by means of Davenport's Lemma. As before, let $\mathcal{F}$ be the standard fundamental domain for the action of $\SL_2(\Z)$ on positive definite binary quadratic forms.  Then
\[
\sum_{\substack{m \in [0, X] \\ \tau_r(m) < N \\ 1-4m \ne \square}} h^+(1-4m) = \# \{(a, b,c) \in \Z^3 \cap \mathcal{F} : b^2 - 4ac = 1-4m \text{ for } m \in [-X, X] \text{ with } \tau_r(m)<N \}  
\]
Because the function $\tau_r(m)$ only depends on $m$ modulo $\gcd(1, \dotsc, r)$, we can use Davenport's lemma as in \eqref{apply davenport} to obtain an estimate of the form
\[
\# \{(a, b,c) \in \Z^3 \cap \mathcal{S} : b^2 - 4ac = 1-4m \text{ for } m \in [-X, X] \text{ with } \tau_r(m)<N \}   = \rho(N, r) \kappa_2 X^{3/2}  + O_{N, r}(X \log X).
\]
where $\rho(N, r)$ is an explictly computable density that goes to $0$ as $r \to \infty$ for fixed $N$.  

We conclude that
\begin{equation}\label{bound flat}
T_1^\flat(X) \le  \rho(N, r) \kappa_2 X^{3/2}  + O_{N, r}(X \log X).
\end{equation}

Combining equations \eqref{bound sharp} and \eqref{bound flat}, and looking at the asymptotic behavior we conclude that

\[
\lim_{n \to \infty} X^{-3/2} T_1(X)   \le  \frac{1}{N} \kappa_1 + \rho(N, r) \kappa_2.
\]
for any $N, r$.  Letting $r \to \infty$ first and then $N \to \infty$, we conclude $\lim_{n \to \infty} X^{-3/2} T_1(X) = 0$.

\end{proof}
\bibliography{genus1}{}

\def\cprime{$'$}
\begin{thebibliography}{10}

\bibitem{bayer-michel-finitude}
E.~Bayer and F.~Michel.
\newblock Finitude du nombre des classes d'isomorphisme des structures
  isometriques enti\`eres.
\newblock {\em Comment. Math. Helv.}, 54(3):378--396, 1979.

\bibitem{bhargava-higher_composition_i}
M.~Bhargava.
\newblock Higher composition laws. {I}. {A} new view on {G}auss composition,
  and quadratic generalizations.
\newblock {\em Ann. of Math. (2)}, 159(1):217--250, 2004.

\bibitem{bhargava-varma}
Manjul Bhargava and Ila Varma.
\newblock The mean number of 3-torsion elements in the class groups and ideal
  groups of quadratic orders.
\newblock {\em Proc. Lond. Math. Soc. (3)}, 112(2):235--266, 2016.

\bibitem{cohen-lenstra}
H.~Cohen and H.~W. Lenstra, Jr.
\newblock Heuristics on class groups of number fields.
\newblock In {\em Number theory, {N}oordwijkerhout 1983 ({N}oordwijkerhout,
  1983)}, volume 1068 of {\em Lecture Notes in Math.}, pages 33--62. Springer,
  Berlin, 1984.

\bibitem{davenport}
H.~Davenport.
\newblock On a principle of {L}ipschitz.
\newblock {\em J. London Math. Soc.}, 26:179--183, 1951.

\bibitem{farber-simple}
M.~Farber.
\newblock Classification of simple knots.
\newblock {\em Uspekhi Mat. Nauk}, 38(5(233)):59--106, 1983.

\bibitem{friedl-powell-blanchfield}
Stefan Friedl and Mark Powell.
\newblock A calculation of {B}lanchfield pairings of 3-manifolds and knots.
\newblock {\em Mosc. Math. J.}, 17(1):59--77, 2017.

\bibitem{gerth_4}
Frank Gerth, III.
\newblock The {$4$}-class ranks of quadratic fields.
\newblock {\em Invent. Math.}, 77(3):489--515, 1984.

\bibitem{gerth_general}
Frank Gerth, III.
\newblock Extension of conjectures of {C}ohen and {L}enstra.
\newblock {\em Exposition. Math.}, 5(2):181--184, 1987.

\bibitem{hooley-quadratic}
C.~Hooley.
\newblock On the {P}ellian equation and the class number of indefinite binary
  quadratic forms.
\newblock {\em J. Reine Angew. Math.}, 353:98--131, 1984.

\bibitem{iwaniec-rossers_sieve}
H.~Iwaniec.
\newblock Rosser's sieve.
\newblock {\em Acta Arith.}, 36(2):171--202, 1980.

\bibitem{kearton-simple}
C.~Kearton.
\newblock Classification of simple knots by {B}lanchfield duality.
\newblock {\em Bull. Amer. Math. Soc.}, 79:952--955, 1973.

\bibitem{kearton}
C.~Kearton.
\newblock Blanchfield duality and simple knots.
\newblock {\em Trans. Amer. Math. Soc.}, 202:141--160, 1975.

\bibitem{levine-unknotting}
J.~Levine.
\newblock Unknotting spheres in codimension two.
\newblock {\em Topology}, 4:9--16, 1965.

\bibitem{levine-polynomial}
J.~Levine.
\newblock Polynomial invariants of knots of codimension two.
\newblock {\em Ann. of Math. (2)}, 84:537--554, 1966.

\bibitem{levine-classification}
J.~Levine.
\newblock An algebraic classification of some knots of codimension two.
\newblock {\em Comment. Math. Helv.}, 45:185--198, 1970.

\bibitem{levine-knot_modules}
J.~Levine.
\newblock Knot modules. {I}.
\newblock {\em Trans. Amer. Math. Soc.}, 229:1--50, 1977.

\bibitem{miller_sp2g}
A.~B. Miller.
\newblock Asymptotics for $\mathrm{Sp}_{2g}$-orbits on quadratic forms.

\bibitem{miller_thesis}
A.~B. Miller.
\newblock Counting simple knots via arithmetic invariant theory.
\newblock {\em Princeton Ph.D. Thesis}, 2014.

\bibitem{siegel-average_measure}
Carl~Ludwig Siegel.
\newblock The average measure of quadratic forms with given determinant and
  signature.
\newblock {\em Ann. of Math. (2)}, 45:667--685, 1944.

\bibitem{trotter-s-equivalence}
H.~Trotter.
\newblock On {$S$}-equivalence of {S}eifert matrices.
\newblock {\em Invent. Math.}, 20:173--207, 1973.

\bibitem{trotter-knot_modules}
H.~Trotter.
\newblock Knot modules and {S}eifert matrices.
\newblock In {\em Knot theory ({P}roc. {S}em., {P}lans-sur-{B}ex, 1977)},
  volume 685 of {\em Lecture Notes in Math.}, pages 291--299. Springer, Berlin,
  1978.

\end{thebibliography}
\bibliographystyle{plain}

\end{document}